\documentclass[12pt,reqno,sort]{elsarticle}

\usepackage{amsfonts,color,amsmath,amssymb,fancyhdr}
\usepackage{amsthm}
\usepackage{graphicx}

\usepackage[colorlinks,linkcolor=blue, anchorcolor=red,citecolor=green]{hyperref}
\usepackage{subfigure}

\def\Box{\vcenter{\vbox{\hrule\hbox{\vrule
     \vbox to 8.8pt{\hbox to 10pt{}\vfill}\vrule}\hrule}}}

\newcommand{\F}{\mathbb{F}}

\newtheorem{thm}{Theorem}[section]
\newtheorem{lemma}[thm]{Lemma}

\numberwithin{equation}{section}
\newtheorem{remark}[thm]{Remark}

\newcommand{\cL}{\mathcal L}

\newcommand{\cS}{\mathcal S}

\definecolor{Purple}{rgb}{0.5,0,0.5}

\def\Aut{{\rm Aut}}

  \def\Tr{{\rm Tr}}

\newcommand{\tup}{\textup}

\begin{document}
\newcommand{\stopthm}{\begin{flushright}
		\(\box \;\;\;\;\;\;\;\;\;\; \)
\end{flushright}}
\newcommand{\symfont}{\fam \mathfam}

\title{New families of flag-transitive linear spaces}

\date{}
\author[add1]{Tao Feng}\ead{tfeng@zju.edu.cn}
\author[add1]{Jianbing Lu\corref{cor1}}\ead{jianbinglu@zju.edu.cn}\cortext[cor1]{Corresponding author}
\address[add1]{School of Mathematical Sciences, Zhejiang University, Hangzhou 310027, Zhejiang, P.R. China}

\begin{abstract}
    In this paper, we construct new families of flag-transitive linear spaces with $q^{2n}$ points and $q^{2}$ points on each line that admit a one-dimensional affine automorphism group. We achieve this by building a natural connection with permutation polynomials of $\mathbb{F}_{q^{2}}$ of a particular form and following the scheme of Pauley and Bamberg in [A construction of one-dimensional affine flag-transitive linear spaces, Finite Fields Appl. 14 (2008) 537-548].
    \newline
	
	\noindent\text{Keywords:} Flag-transitive; Linear space; Permutation polynomial; Irreducible polynomial; $2$-design; $t$-spread.
	
\end{abstract}	

\maketitle

\section{Introduction}\label{introduction}

A  finite \emph{linear space} $\mathcal{L}$ is a point-line incidence structure with finitely many points such that any two points lie on exactly one common line, any point lies on  at least two lines and any line is incident  with at least two points. It is \emph{non-trivial} if some line has more than two points, and it is nondegenerate if it has four points no three of which are collinear. All the linear spaces in consideration are finite, nontrivial and nondegenerate. An automorphism of a linear space $\cL$ consists of a permutation of points and a permutation of lines that preserve incidence. We write $\Aut(\cL)$ for the full automorphism group of $\cL$, and call its subgroups as automorphism groups. A \emph{flag} of  $\mathcal{L}$ is an incident point-line pair. We say that an automorphism group $G$ is flag-transitive if it acts transitively on the flags of $\cL$. In such a case, each line has the same number of points which we denote by $k$, and $\cL$ forms a $2-(v,k,1)$ design with $v$ the total number of points of $\cL$.

There have been extensive works on the classification of finite flag-transitive linear spaces. Suppose that $\cL$ is such a linear space with a flag-transitive automorphism group $G$. By \cite{ref6} $G$ is point primitive, and by \cite{ref7} $G$ is either almost simple or of affine type. The classification in the almost simple case is accomplished in \cite{ref8}. In the case where $G$ is  of affine type, there is a point regular subgroup $T$ of order $p^d$ that is elementary abelian. We regard $T$ as an additive group and identify the points of $\cL$ with $T$. We have $G=TG_0$ with $G_0\le \Gamma\textup{L}_d(p)$. If $d\ge2$ and $G_0$ is not contained in $\Gamma\textup{L}_1(p^d)$, the classification is accomplished in \cite{ref9}. In the remaining case where $G_0$ is  a subgroup of $\Gamma\textup{L}_1(p^d)$, there are a variety of examples, cf. \cite{ref1,ref11,ref12,ref13}, but a complete classification is still out of reach.

In \cite{ref1}, the authors presented a method of deriving one-dimensional affine flag-transitive linear spaces where the input is an irreducible polynomial over $\F_{q^2}$ that satisfies certain property. In particular, they showed that there are flag-transitive linear space with $p^{2p}$ points and $p^{2}$ points on each line for each odd prime $p$. The purpose of the present paper is to build a connection with  permutation polynomials of the form $x^{r}h(x^{(q-1)/s})$. This will lead us to new flag-transitive linear spaces by making use of the known results on permutation polynomials.
\begin{thm}\label{main}
Let $q$ be a prime power and $d>1$ be an odd divisor of $q+1$. Let $u$ be a proper divisor of $d$, $t$ be a positive integer and set $n=d^{t}u$. Then there exist flag-transitive (non-Desarguesian) linear spaces with a one-dimensional affine automorphism group, and $q^{2n}$ points and $q^{2}$ points on each line.
\end{thm}
In characteristic $3$, we have an extra family from an irreducible polynomial of degree $3$ over $\F_{q^2}$.
\begin{thm}\label{second}
Let $q=3^{k}$. Then there exist flag-transitive (non-desarguesian) linear spaces with a one-dimensional affine automorphism group, and $q^{6}$ points and $q^{2}$ points on each line.
\end{thm}

The proofs of Theorem \ref{main} and Theorem \ref{second} are given in Section \ref{s3} and Section \ref{s4}, using some preliminary results about permutation polynomials presented in Section \ref{s2}. In \cite{ref1}, the authors described a procedure called inflation to obtain new linear spaces over an extension field with the same irreducible polynomial, which generalizes a similar procedure in \cite{ref12}. The inequivalence of our new linear spaces with the known ones and the fact that they do not arise from inflation can be established by following the same lines as in \cite{ref1}.

The paper is organized as follows. In Section \ref{s2}, we present some preliminary results on linear spaces and permutation polynomials. We present the proof of Theorem \ref{main} in Section \ref{s3} and the proof of Theorem \ref{second} in Section \ref{s4}. We consider the equivalence issue in Section \ref{s5}.

\section{Preliminaries}\label{s2}

\subsection{$k$-spreads and flag-transitive linear spaces}\label{s21}

Let $V$ be an $n$-dimensional vector space over a finite field $\mathbb{F}_{q}$. A $k$-spread of $V$ is a set of $(k+1)$-dimensional subspaces of $V$ which partition the nonzero vectors. A $k$-spread exists if and only if $k+1\mid n$. For instance, if $k+1$ divides $n$ and  we regard $V'=\mathbb{F}_{q^{k+1}}^{n/(k+1)}$ as a $n$-dimensional vector space $V$ over $\mathbb{F}_q$, then the $1$-dimensional $\mathbb{F}_{q^{k+1}}$-subspaces of $V'$ form the Desarguesian $k$-spread of $V$. We say that a $k$-spread $\mathcal{S}$ is transitive if the stabiliser of $\mathcal{S}$ in $\Gamma \tup{L}_{n}(q)$ acts transitively on the elements of $\mathcal{S}$. The  Andr{\'e}/Bruck-Bose construction (see \cite{ref14,ref16}) gives a method to produce a linear space from a $k$-spread $\mathcal{S}$ of a vector space $V$ as follows: the points of the linear space are the vectors of $V$ and the lines  are  the sets $S+v$, where $S\in\mathcal{S}$ and $v\in V$. The resulting linear space is flag-transitive if and only if $\cS$ is transitive. Moreover, a flag-transitive linear space whose automorphism group is of affine type must arise in this way.  By \cite{ref11}, a  linear space $\cL$ with a flag-transitive automorphism group $G$ is either a Desarguesian affine space, a Hering plane of order $27$, one of two sporadic linear spaces constructed by Hering that are not planes \cite{ref10}, or $G$ is a subgroup of $\textup{A}\Gamma\textup{L}_{1}(p^{n})$ with $p$ prime.

In \cite{ref1}, Pauley and Bamberg gave a construction of one-dimensional affine flag-transitive linear spaces via the Andr{\'e}/Bruck-Bose construction applied to transitive $1$-spreads of vector space. To be specific, we take $V=\mathbb{F}_{q^{2m}}$ and $b\in V$, set $\ell_b=\{x-bx^{q}:\,x\in\mathbb{F}_{q^2}\}$ with $b^{q+1}\ne 1$, and let $C$ be the subgroup of $\Gamma\textup{L}(V)$ that consists of the linear transformations $x\mapsto ux$ of $V$ with $u^{(q^{2m}-1)/(q+1)}=1$. The next result describes the necessary and sufficient conditions for $\ell_b^C$ to form a transitive $1$-spread.
\begin{lemma}\label{spread}
\cite[Theorem 1]{ref1} Take notation as above,  let $h(x)$ be the minimal polynomial of $b$ over $\mathbb{F}_{q^{2}}$, and let $d$ be the degree of $h(x)$. Then $\ell_b^C$ is a transitive $1$-spread of $V$ if and only if for any nonzero $x,y\in\mathbb{F}_{q^{2}}$ we have that
\begin{equation}\label{permutate}
\frac{x^{m}h(x^{q-1})^{m/d}}{y^{m}h(y^{q-1})^{m/d}}\in\mathbb{F}_{q}\text{\quad implies that} \quad\frac{x}{y}\in \mathbb{F}_{q}.
\end{equation}
Moreover, $\ell_b^C$ is Desarguesian if and only if $b$ is in $\F_{q^2}$.
\end{lemma}
The authors show that in the case $q=p$ is an odd prime and $m=p$, the polynomial
\begin{equation}\label{eqn_PBexamp}
h(x)=\frac{x^{p+1}-1}{x-1}-2=x^{p}+x^{p-1}+\cdot\cdot\cdot+x-1
\end{equation}
is irreducible over $\mathbb{F}_{p^{2}}$ and satisfies condition \eqref{permutate}. Correspondingly there is a new flag-transitive linear space with $p^{2p}$ points and $p^{2}$ points on each line for each odd prime $p$.

\subsection{Permutation polynomials of the form $x^{r}h(x^{(q-1)/s})$}

A polynomial $f(x)\in\mathbb{F}_{q}[x]$ is a permutation polynomial of $\mathbb{F}_{q}$ if the associated  function $f:\,c\mapsto f(c)$ of $\mathbb{F}_{q}$  is a permutation of $\mathbb{F}_{q}$. Wan and Lidl \cite{ref15} initiated the systematic study of permutation polynomials of $\mathbb{F}_{q}$ of the form $x^{r}h(x^{(q-1)/s})$, where $s$ divides $q-1$, $r>0$ and $h(x)\in\mathbb{F}_{q}[x]$. A criterion for a polynomial in such form to be a permutation polynomial was given in \cite{ref15}.  There have been extensive studies on permutation polynomials of this form in recent years, and we refer the interested reader to the comprehensive survey \cite{ref5}. We shall need the following two results in the sequel.
\begin{lemma}\label{tripp}\cite[Theorem 1.2]{ref2} Let $q$ be a prime power, let $d,k$ be integers with $d>0$ and $k\geq 0$, and let $\beta,\delta\in\mathbb{F}_{q^{2}}$ satisfy $\beta^{q+1}=1$ and $\delta\notin\mathbb{F}_{q}$. Then
$$f(x):=x^{d+k(q+1)}\cdot\left(\left(\delta x^{q-1}-\beta\delta^{q}\right)^{d}-\delta\left(x^{q-1}-\beta\right)^{d}\right)$$
permutes $\mathbb{F}_{q^{2}}$ if and only if $\gcd(d(d+2k),q-1)=1$.
\end{lemma}

\begin{lemma}\label{quadripp}
\cite[Theorem 3.2]{ref3} Let $q=3^{k}$. Then for $a,c\in\mathbb{F}_{q}^{*}$, the quadrinomial
\[
f(x):=x^{3}+ax^{q+2}-ax^{2q+1}+cx^{3q}
\]
permutes $\mathbb{F}_{q^{2}}$ if any one of the following conditions is satisfied: (1) $c=a\neq-1$ and $a^{(q-1)/2}=1$; (2) $c=a-1$ and $(-a)^{(q-1)/2}=1$; (3) $c=1-a, a\neq-1$ and $k$ is even; (4) $c=1$.
\end{lemma}
We observe that the function $f(x)$ in Lemma \ref{quadripp} is of the form $f(x)=x^{3}h(x^{q-1})$, where $h(x)=cx^{3}-ax^{2}+ax+1$.

\section{Proof of Theorem \ref{main}}\label{s3}
Our key observation is the following lemma which builds a connection between the condition (\ref{permutate}) in Lemma \ref{spread} and permutation polynomials.
\begin{lemma}\label{coprime}
Let $q$ be a prime power and $d$ be a positive integer such that $\gcd(d,q-1)=1$. If $x^{d}h(x^{q-1})$ is a permutation polynomial of $\mathbb{F}_{q^{2}}$, then for all nonzero $x,y\in\mathbb{F}_{q^{2}}$ we have that
\[
\frac{x^{d}h(x^{q-1})}{y^{d}h(y^{q-1})}\in\mathbb{F}_{q}\text{\quad implies that} \quad\frac{x}{y}\in \mathbb{F}_{q}.
\]
\end{lemma}
\begin{proof}
Suppose $x^{d}h(x^{q-1})=\lambda y^{d}h(y^{q-1})$ for some $\lambda\in\mathbb{F}_{q}^{\ast}$ and $x,y\in\mathbb{F}_{q^{2}}^*$. Since $\gcd(d,q-1)=1$, there exists $c\in\mathbb{F}_{q}^{\ast}$ such that $\lambda=c^{d}$. Set $z=cy$. Then $z^{d}=c^{d}y^{d}=\lambda y^{d}$ and $h(z^{q-1})=h(c^{q-1}y^{q-1})=h(y^{q-1})$. Therefore, we have $x^{d}h(x^{q-1})=z^{d}h(z^{q-1})$, which leads to $x=z=cy$ by the fact that $x^{d}h(x^{q-1})$ is a permutation polynomial of $\mathbb{F}_{q^{2}}$. This completes the proof.
\end{proof}

\begin{remark}
We now take a look at the polynomial $h(x)$ in \eqref{eqn_PBexamp}. For $x\in\mathbb{F}_{p^2}$, we have
\[
x^{p}h(x^{p-1})=
\begin{cases}
-2x^{p},\;&\text{if }x\not\in\mathbb{F}_p^*,\\
-x,\;&\text{if }x\in\mathbb{F}_p^*.
\end{cases}
\]
It is routine to check that $x^{p}h(x^{p-1})$ permutes $\mathbb{F}_{p^{2}}$, and so falls in the category of Lemma \ref{coprime}.
\end{remark}

\begin{lemma}\label{gcd}
Let $q$ be a prime power and $d>1$ be an odd divisor of $q+1$. Let $i>0$ be an even integer. Then
\begin{equation}\label{eqn_gcd}
\gcd\left(d^{t}(q+1), q^{i\cdot d^{t-1}}-1\right)=(q+1)\cdot d^{t-1}\cdot\gcd(d,i)
\end{equation}
for any positive integer $t$. Furthermore, let $N=d^{t}(q+1)$ and $\tup{ord}_{N}(q)$ be the order of $q$ in the multiplicative group $\mathbb{Z}_{N}^{\ast}$ of the residue ring $\mathbb{Z}_{N}$. Then $\tup{ord}_{N}(q)=2d^{t}$.
\end{lemma}
\begin{proof}
We prove the first claim by induction on $t$. In the case $t=1$, we have
\begin{align*}
\gcd(d(q+1),q^{i}-1)=&\gcd\left(d(q+1),(q+1)\cdot\sum_{l=0}^{i-1}(-1)^{l+1}q^{l}\right)\\
=&(q+1)\cdot\gcd\left(d,\sum_{l=0}^{i-1}(-1)^{l+1}q^{l}\right).
\end{align*}
Since $q\equiv-1$ (mod $d$), we have $\sum_{l=0}^{i-1}(-1)^{l+1}q^{l}\equiv-i$ (mod $d$). The claim for $t=1$ then follows. Suppose that it holds for $t=k\geq1$, i.e.,
\begin{equation}\label{0}
\gcd\left(d^{k}(q+1), q^{i\cdot d^{k-1}}-1\right)=(q+1)\cdot d^{k-1}\cdot\gcd(d,i).
\end{equation}
Then $d^{k-1}(q+1)\mid q^{i\cdot d^{k-1}}-1$ and therefore $d^{k}(q+1)\mid (q^{i\cdot d^{k-1}}-1)^{2}$. Note that
\begin{center} $\sum_{l=2}^{d}(l-1)\cdot q^{i\cdot(d^{k}-l\cdot d^{k-1})}\equiv\sum_{l=2}^{d}(l-1)\equiv\frac{d(d-1)}{2}\equiv0$ (mod $d$).
\end{center}
Thus we have
\begin{equation}\label{1}
d^{k+1}(q+1)\Bigm| \left(q^{i\cdot d^{k-1}}-1\right)^{2}\cdot\left(\sum\limits_{l=2}^{d}(l-1)\cdot q^{i\cdot(d^{k}-l\cdot d^{k-1})}\right).
\end{equation}
Since  $q^{i\cdot d^{k}}-1=\left(q^{i\cdot d^{k-1}}-1\right)\cdot\sum_{l=1}^{d}q^{i\cdot(d^{k}-l\cdot d^{k-1})}$ and
\begin{equation*}\label{12}
\sum_{l=1}^{d}q^{i\cdot(d^{k}-l\cdot d^{k-1})}=\left(q^{i\cdot d^{k-1}}-1\right)\cdot\left(\sum_{l=2}^{d}(l-1)\cdot q^{i\cdot(d^{k}-l\cdot d^{k-1})}\right)+d,
\end{equation*}
we deduce that
\begin{equation}\label{2}
q^{i\cdot d^{k}}-1=\left(q^{i\cdot d^{k-1}}-1\right)^{2}\cdot\left(\sum_{l=2}^{d}(l-1)\cdot q^{i\cdot(d^{k}-l\cdot d^{k-1})}\right)+\left(q^{i\cdot d^{k-1}}-1\right)\cdot d.
\end{equation}
From (\ref{1}) and (\ref{2}), we have that $q^{i\cdot d^{k}}-1\equiv (q^{i\cdot d^{k-1}}-1)\cdot d$ (mod $d^{k+1}(q+1)$). By (\ref{0}),
\begin{align*}
\gcd\left(d^{k+1}(q+1),q^{i\cdot d^{k}}-1\right)
=&\gcd\left(d^{k+1}(q+1),\left(q^{i\cdot d^{k-1}}-1\right)\cdot d\right)\\
=&d\cdot\gcd\left(d^{k}(q+1),q^{i\cdot d^{k-1}}-1\right)=(q+1)\cdot d^{k}\cdot\gcd(d,i).
\end{align*}
This completes the proof of the first part of the lemma.

We next determine $e:=\tup{ord}_{N}(q)$, where $N=d^t(q+1)$ with $t>0$. Since $q^{e}-1\equiv (-1)^{e}-1\equiv 0$ (mod $d$), we deduce that $e$ is even. From $N\mid q^{e}-1$, we have $d^{t}\mid(q^{e}-1)/(q+1)$. Note that
\begin{center}
$(q^{e}-1)/(q+1)=\sum_{l=0}^{e-1}(-1)^{l+1}q^{l}\equiv-e$ (mod $d$).
\end{center}
Thus we have $d\mid e$. On the other hand, $$\gcd\left(N,q^{2d^{t}}-1\right)=(q+1)\cdot d^{t-1}\cdot\gcd(d,2d)=N,$$
so $e$ divides $2d^{t}$.  Assume to the contrary that $e<2d^{t}$. We write $e=d^{s}v$, where $1\leq s\leq t-1$ and $d\nmid v$. Then $q^{e}-1\mid q^{d^{t-1}v}-1$ and so $N\mid q^{d^{t-1}v}-1$. On the other hand, by \eqref{eqn_gcd} we have $$\gcd\left(N,q^{d^{t-1}v}-1\right)=(q+1)\cdot d^{t-1}\cdot\gcd(d,v)<N,$$
which is a contradiction. This completes the proof.
\end{proof}

\begin{lemma}\label{power}
Let $q$ be a prime power and $d>1$ be an odd divisor of  $q+1$. Let $u$ be a proper divisor of $d$, $t$ be a positive integer and set $n=d^{t}u$. Let $\delta$ be an element of order $q+1$ in $\mathbb{F}_{q^{2}}^{\ast}$. Then the polynomial
$$g_{n}(x):=\frac{(\delta x-1)^{n}-\delta(x-\delta)^{n}}{\delta^{n}-\delta}$$
has coefficients in $\F_q$ and is an irreducible polynomial in $\mathbb{F}_{q^{2}}[x]$ of degree $n$.
\end{lemma}
\begin{proof}
Since $d$ divides $q+1$ and $n$, we see that $q+1$ does not divides $n-1$ and so  $\delta^{n}-\delta\neq0$. Hence $g_{n}(x)$ has degree $n$.
We first prove that $g_{n}(x)$ has coefficients in $\mathbb{F}_q$. To see this, it suffices to show that $g_{n}(x^{q})=g_{n}(x)^{q}$. We compute that
\begin{align*}
g_{n}(x)^{q}
=&\frac{(\delta^{q} x^{q}-1)^{n}-\delta^{q}(x^{q}-\delta^{q})^{n}}{\delta^{qn}-\delta^{q}}
=\frac{(\delta^{-1} x^{q}-1)^{n}-\delta^{-1}(x^{q}-\delta^{-1})^{n}}{\delta^{-n}-\delta^{-1}}\\
=&\frac{\delta( x^{q}-\delta )^{n}-(\delta x^{q}-1)^{n}}{\delta-\delta^{n}}
=g_{n}(x^{q}).
\end{align*}
Here, we have used the fact $\delta^{q+1}=1$.

Next, we prove that $g_{n}(x)$ is irreducible in $\mathbb{F}_{q^2}[x]$. Since the degree $n$ of $g(x)$ is relatively prime to $2$, it suffices to show that $g_{n}(x)$ is irreducible in $\mathbb{F}_{q}[x]$. Assume to the contrary that $g_{n}(x)$ is reducible in $\mathbb{F}_{q}[x]$. Let $h(x)$ be an irreducible factor of $g_{n}(x)$ in $\mathbb{F}_{q}[x]$ of smallest degree, and write $i=\deg(h(x))$. Then $1\leq i\leq\frac{n-1}{2}$, and the roots of $h(x)$ are in $\mathbb{F}_{q^{i}}$. Suppose that $x_{1}$ is a root of $h(x)$ in $\mathbb{F}_{q^{i}}$. From $g_n(x_1)=0$, we deduce that $\delta=\left(\frac{\delta x_{1}-1}{x_{1}-\delta}\right)^{n}$. Write $c=\frac{\delta x_{1}-1}{x_{1}-\delta}$, which lies in $\mathbb{F}_{q^{\tup{lcm}(2,i)}}$ and so $c^{q^{2i}-1}=1$. Since $\delta$ has order $q+1$ and $\delta=c^n$, we deduce that $c^{n(q+1)}=1$ and  $o(c)=m(q+1)$ for some divisor $m$ of $n$. We thus deduce that $m(q+1)$ divides $\gcd(n(q+1),q^{2i}-1)$. If we can show that $\gcd(n(q+1),q^{2i}-1)<n(q+1)$, then $m<n$ and we can deduce a contradiction as follows. Let $r$ be a prime divisor of $n/m$. Since $n=d^tu$ with $u\mid d$, we have $r\mid d$ and so $r\mid q+1$. It follows that $\gcd(n,m(q+1))=m\cdot\gcd(n/m,q+1)>m$. The order of $\delta=c^n$ is $\frac{o(c)}{\gcd(o(c),n)}=\frac{m(q+1)}{\gcd(n,m(q+1))}<q+1$: a contradiction.

It now remains to show that $D_i:=\gcd(n(q+1),q^{2i}-1)<n(q+1)$ for any $1\leq i\leq \frac{n-1}{2}$. If $d^t$ divides $i$, then $u>1$ and $2i=d^tv$ with $1\le v\le u-1$. Since $n=d^{t}u$ is a divisor of $d^{t+1}$,  $D_i$ divides $\gcd\left(d^{t+1}(q+1),q^{d^{t}\cdot v}-1\right)$. By Lemma \ref{gcd}, the latter number equals $(q+1)\cdot d^{t}\cdot\gcd(d,v)$, which is smaller than $n(q+1)$ as desired. It remains to consider the case where $d^t$ does not divide $i$. We have that $D_i$ divides $u\cdot\gcd(N,q^{2i}-1)$ with $N=d^t(q+1)$ in this case. By Lemma \ref{gcd},
$\textup{ord}_N(q)=2d^t$, so $\gcd(N,q^{2i}-1)<N$ by the fact that $d^t$ does not divide $i$. Therefore, $D_i<uN=n(q+1)$ as desired. This completes the proof.
\end{proof}

\noindent\textbf{Proof of Theorem \ref{main}.} Take the same notation as in Lemma \ref{power}. The polynomial $g_n(x)$ is irreducible over $\mathbb{F}_{q^2}$ of degree $n$. Since $d$ is an odd divisor of $q+1$ and $n=d^{t}u$, we have $\gcd(n^{2},q-1)=1$. By Lemma \ref{tripp} with $k=0$ and $\beta=\delta$, $x^{n}g_{n}(x^{q-1})$ permutates $\mathbb{F}_{q^{2}}$. Therefore, $g_{n}(x)$ satisfies condition (\ref{permutate}) by Lemma \ref{coprime}. The claim then follows from Lemma \ref{spread}. This completes the proof.

\section{Proof of Theorem \ref{second}}\label{s4}

Let $\Tr_{3^{k}/3}$ be the trace function from $\mathbb{F}_{3^{k}}$ to $\mathbb{F}_{3}$, i.e., $\Tr_{3^{k}/3}(x)=x+x^3+\cdots+x^{3^{k-1}}$.

\begin{lemma}\label{cubic}
\cite[Corollary 1.23]{ref4} Let $\mathbb{F}=\mathbb{F}_{3^{k}}$ and $b,c\in\mathbb{F}$. If $b=s^{2}$ with $s$ in $\mathbb{F}^{\ast}$ and $\Tr_{3^{k}/3}(c/s^{3})\neq0$, then $h(x)=x^{3}-bx-c$ has no roots in $\mathbb{F}$.
\end{lemma}

\begin{lemma}\label{triirr}
Let $a\in\mathbb{F}_{3^{k}}^{\ast}$ such that $\Tr_{3^{k}/3}(a^{-1})\neq0$, then $h(x)=x^{3}+ax^{2}-ax+1$ is irreducible in $\mathbb{F}_{3^{2k}}[x]$.
\end{lemma}
\begin{proof}
Write $q=3^k$. Since $h(x)$ has coefficients in $\mathbb{F}_{q}$ and has odd degree $3$, it suffices to show that $h(x)$ is irreducible over $\mathbb{F}_{q}$, or equivalent it has no roots in $\mathbb{F}_{q}$.  Note that $h(x-1)=x^{3}+ax^{2}-a$ for $x\in\F_{q}$ and $h(-1)=-a\neq0$. We have $h(1/x-1)=\frac{-a}{x^{3}}(x^{3}-x-a^{-1})$ for $x\in\F_{q}^*$.
Set $s=1$ and $c=a^{-1}$ in Lemma \ref{cubic}, and we see that  $x^{3}-x-a^{-1}$ has no roots in $\mathbb{F}_{q}$. The claim now follows.
\end{proof}

\noindent\textbf{Proof of Theorem \ref{second}.}
\begin{proof}
Let $h(x)$ be as in Lemma \ref{triirr}, which is irreducible over $\F_{q^2}$ with $q=3^k$.  Moreover, Lemma \ref{quadripp} says that $x^{3}h(x^{q-1})=x^{3q}+ax^{2q+1}-ax^{q+2}+x^{3}$ permutes $\mathbb{F}_{q^2}$. Since $\gcd(3,q-1)=1$, the irreducible polynomial $h(x)$ satisfies condition (\ref{permutate}) with $m=d=3$ by Lemma \ref{coprime}. The claim then follows from Lemma \ref{spread}.
\end{proof}

\section{The equivalence issue}\label{s5}

A flag-transitive linear space $\mathcal{L}_{1}$, whose points form the field $\mathbb{F}_{q^{mm^{\prime}n}}$, is an \emph{inflation} of another flag-transitive linear space $\mathcal{L}_{2}$, whose points form the field $\mathbb{F}_{q^{mn}}$, if the lines of $\mathcal{L}_{2}$ are just those lines of $\mathcal{L}_{1}$ which are wholly contained in $\mathbb{F}_{q^{mn}}$. The following proposition describes how to obtain inflations of a linear space arising from Lemma \ref{spread} by \emph{keeping the polynomial the same but varying the field}.
\begin{lemma}\label{inflation}
\cite[Proposition 3]{ref1} Let $h(x)$ be an irreducible polynomial over $\mathbb{F}_{q^{2}}$ of degree $d\geq2$. Suppose $h(x)$ satisfies condition (\ref{permutate}) in the field $\mathbb{F}_{q^{2mm^{\prime}}}$. Then the following are equivalent:
\begin{enumerate}
\item[(i)] there is a flag-transitive linear space arising from $h(x)$ in the field $\mathbb{F}_{q^{2m}}$ and the flag-transitive linear space arising from $h(x)$ in the field $\mathbb{F}_{q^{2mm^{\prime}}}$ is an inflation of it;
\item[(ii)] the flag-transitive linear space arising from $h(x)$ in the field $\mathbb{F}_{q^{2mm^{\prime}}}$ is isomorphic to an inflation of some flag-transitive linear space with point set $\mathbb{F}_{q^{2m}}$;
\item[(iii)] $d$ divides $m$, and $m^{\prime}$ is coprime to $q+1$;
\item[(iv)] $h(x)$ satisfies condition (\ref{permutate}) in the field $\mathbb{F}_{q^{2m}}$.
\end{enumerate}
\end{lemma}

The linear spaces constructed in Theorem \ref{main} and Theorem \ref{second} are not Desarguesian by Lemma \ref{spread}, and are not inflations of one-dimensional flag-transitive linear spaces by (iii) of Lemma \ref{inflation} upon direct check. To compare with Kantor's constructions in \cite{ref12}, it suffices to compare with those of Type 4 by discussion in \cite{ref1}. The number of points in our constructions is $q^{2n}$ and the number of points on a line is $q^{2}$, where either (a) $n=d^{t}u$, $u\mid d$, $d\mid q+1$ and $d$ odd, or (b) $q=3^k$ and $n=3$. The number of points in a linear space of Kantor's constructions of Type $4$ is $q^{mn'}$ and the number of points on a line is $q^{n'}$ with $n'>1$ and $m$ dividing $q-1$. An isomorphism would imply $n'=2$ and $m=n$, in which case $n\mid q-1$: a contradiction. Therefore, our constructions are not isomorphic to those of Kantor \cite{ref12}.  The corresponding irreducible polynomials $h(x)$ in the proofs of Theorem \ref{main} and Theorem \ref{second} are not of the form $P(x^s)$ with $s\ge 3$ an odd integer coprime to $q+1$, so does not arise from Example 2 of \cite{ref1}. Our construction works for more values of $n$ compared with \cite{ref1} and \cite{ref13}, and so indeed we obtain flag-transitive linear spaces with new parameters.

\section{Concluding remarks}
In this paper, we build a connection between permutation polynomials of the form $x^{r}h(x^{(q-1)/s})$ and one-dimensional affine flag-transitive linear space. This leads to new infinite families of flag-transitive linear spaces by following the scheme described in \cite{ref1}. We show that they have new parameters and are not isomorphic to the known constructions. Our results indicate that there should be more such linear spaces that are associated to interesting permutation polynomials.

\vspace*{10pt}

\noindent\textbf{Acknowledgement.} This work was supported by National Natural Science Foundation of China under Grant No.11771392.

\begin{center}
	\scriptsize
	\setlength{\bibsep}{0.5ex}  
		\linespread{0.5}
	\bibliographystyle{plain}

\end{center}

\end{document}